\documentclass[12pt]{article}
\usepackage{amsmath,amsthm,amssymb}
\usepackage{mathrsfs}
\DeclareMathOperator{\qdim}{qdim}

\DeclareMathOperator{\ch}{ch}
\DeclareMathOperator{\tr}{tr}
\DeclareMathOperator{\diag}{diag}
\DeclareMathOperator{\glob}{glob}
\begin{document}
\input amssym.def
\setcounter{equation}{0}
\newcommand{\wt}{{\rm wt}}
\newcommand{\spa}{\mbox{span}}
\newcommand{\Res}{\mbox{Res}}
\newcommand{\End}{\mbox{End}}
\newcommand{\Ind}{\mbox{Ind}}
\newcommand{\Hom}{\mbox{Hom}}
\newcommand{\Mod}{\mbox{Mod}}
\newcommand{\m}{\mbox{mod}\ }
\renewcommand{\theequation}{\thesection.\arabic{equation}}
\numberwithin{equation}{section}

\def \End{{\rm End}}
\def \Aut{{\rm Aut}}
\def \Z{\mathbb Z}
\def \H{\mathbb H}
\def \M{\Bbb M}
\def \C{\mathbb C}
\def \R{\mathbb R}
\def \Q{\mathbb Q}
\def \N{\mathbb N}
\def \ann{{\rm Ann}}
\def \<{\langle}
\def \o{\omega}
\def \O{\Omega}
\def \Or{\cal O}
\def \M{{\cal M}}
\def \1t{\frac{1}{T}}
\def \>{\rangle}
\def \t{\tau }
\def \a{\alpha }
\def \e{\epsilon }
\def \l{\lambda }
\def \L{\Lambda }
\def \g{\gamma}
\def \b{\beta }
\def \om{\omega }
\def \o{\omega }
\def \ot{\otimes}
\def \cg{\chi_g}
\def \ag{\alpha_g}
\def \ah{\alpha_h}
\def \ph{\psi_h}
\def \S{\cal S}
\def \nor{\vartriangleleft}
\def \V{V^{\natural}}
\def \voa{vertex operator algebra\ }
\def \voas{vertex operator algebras}
\def \v{vertex operator algebra\ }
\def \1{{\bf 1}}
\def \be{\begin{equation}\label}
\def \ee{\end{equation}}
\def \qed{\mbox{ $\square$}}
\def \pf {\noindent {\bf Proof:} \,}
\def \bl{\begin{lem}\label}
\def \el{\end{lem}}
\def \ba{\begin{array}}
\def \ea{\end{array}}
\def \bt{\begin{thm}\label}
\def \et{\end{thm}}
\def \br{\begin{rem}\label}
\def \er{\end{rem}}
\def \ed{\end{de}}
\def \bp{\begin{prop}\label}
\def \ep{\end{prop}}
\def \p{\phi}
\def \d{\delta}
\def \irr{\rm irr}

\newtheorem{th1}{Theorem}
\newtheorem{ree}[th1]{Remark}
\newtheorem{thm}{Theorem}[section]
\newtheorem{prop}[thm]{Proposition}
\newtheorem{coro}[thm]{Corollary}
\newtheorem{lem}[thm]{Lemma}
\newtheorem{rem}[thm]{Remark}
\newtheorem{de}[thm]{Definition}
\newtheorem{hy}[thm]{Hypothesis}
\newtheorem{conj}[thm]{Conjecture}
\newtheorem{ex}[thm]{Example}

\begin{center}
{\Large {\bf On Orbifold Theory}}\\
\vspace{0.5cm}

Chongying Dong\footnote
{Supported by a NSF grant DMS-1404741 and China NSF grant 11371261}
\\
School of Mathematics, Sichuan University,
 Chengdu 610064 China \& \\
 Department of Mathematics, University of
California, Santa Cruz, CA 95064 USA \\
Li Ren\footnote{Supported by China NSF grant 11301356}\\
 School of Mathematics,  Sichuan University,
Chengdu 610064 China\\
Feng Xu\\
Department of Mathematics, University of California, Riverside, CA
92521, USA \& Department of Mathematics, Dalian University of Technology, China

\end{center}

\begin{abstract} Let $V$ be a simple vertex operator algebra and $G$ a finite automorphism group of $V$ such that $V^G$ is regular. It is proved that every irreducible $V^G$-module occurs in an irreducible $g$-twisted $V$-module for some $g\in G.$ Moreover, the quantum dimensions of each irreducible $V^G$-module are determined
and a global dimension formula for $V$ in terms of twisted modules is obtained.
\end{abstract}

\section{Introduction}

This paper deals with general orbifold theory. More precisely, we are trying to understand the representation theory for the fixed point vertex operator subalgebra of a vertex operator algebra under the action of a finite automorphism group. Let $V$ be a rational  vertex operator algebra and $G$ a finite automorphism group of $V.$ The well known orbifold theory conjecture says that $V^G$ is rational and every irreducible $V^G$-module occurs in an irreducible $g$-twisted $V$-module for some $g\in G.$ Proving this conjecture is definitely  the major task in orbifold theory.

The appearance of the twisted modules is the main feature of the orbifold theory. The twisted modules for lattice vertex operator algebras were constructed and studied in \cite{FLM1}, \cite{L}, \cite{FLM2}, \cite{DL2}. The abstract twisted representation theory such as $g$-rationality and classification of irreducible $g$-twisted $V$-modules were investigated in \cite{DLM4} in terms of Zhu's algebra \cite{Z}, and extended further in \cite{DLM5}, \cite{DLM6}, \cite{MT}. The number of inequivalent irreducible $g$-twisted $V$-modules was determined if $V$ is $g$-rational and $C_2$-cofinite  using the modular invariance of trace functions in orbifold theory \cite{DLM7}.

Motivated by the orbifold theory conjecture, decomposition of $V$ into a direct sum of irreducible $V^G$-modules
was initiated in \cite{DM} and \cite{DLM2}. It was shown that $G$ and $V^G$ form a dual pair on $V$ in the sense of \cite{H1}-\cite{H2} and  a Schur-Weyl duality was obtained. The decomposition of an arbitrary irreducible $g$-twisted $V$-module  into a direct sum of $V^G$-modules was achieved in \cite{DY} and \cite{MT}. In other words,
the irreducible $V^G$-modules occurring in all irreducible twisted modules were classified.

We establish in this paper that if $V^G$ is rational, $C_2$-cofinite and the weight of any irreducible $g$-twisted $V$ module except $V$ itself is positive then every irreducible $V^G$-module is a $V^G$-submodule of irreducible
$g$-twisted $V$-module for some $g\in G.$ This result reduces the orbifold theory conjecture to the rationality of $V^G.$ The main tool  we use to prove the result is the modular invariance of trace functions  \cite{Z}, \cite{DLM7}. This explains why the $C_2$-cofiniteness is necessary. The assumption that any irreducible $g$-twisted $V$-module has a positive weight except the vertex operator algebra itself holds for the most well known rational vertex operator algebras.  This assumption will allow us to use the tensor product and related results given in \cite{H}, \cite{DLN}.

We compute the quantum dimensions of any irreducible $g$-twisted $V$-module and any irreducible $V^G$-module. The quantum dimension for a $V$-module was defined and studied in \cite{DJX} to have a full  Galois theory $V^G\subset V$ \cite{DM}, \cite{HMT} where the quantum dimension $\qdim_{V^G}V$ plays the role that the ordinary dimension plays in the classical Galois correspondence. The quantum dimensions of the irreducible modules are the important invariants of $V$ and the product formula  $\qdim_V(M\boxtimes N)=\qdim_VM\dot\qdim_VN$
for any $V$-modules $M,N$ is very useful in computing the fusion rules. An explicit relation between
the quantum dimension of an irreducible $g$-twisted $V$-module $M$ and the quantum dimension of an irreducible $V^G$-submodule of $M$ is given.

We also find several interesting formulas on the global dimension of
$V.$ The global dimension is defined to be the sum of the squares of
the quantum dimensions of the inequivalent irreducible $V$-modules \cite{DJX}. The
global dimension, in fact, is  the Frobenius-Perron dimension of the
$V$-module category \cite{ENO}, \cite{H} and \cite{DJX}. It was
established in \cite{DJX} that the global dimension $\glob(V)$ is
equal to $\frac{1}{S_{V,V}^2}$ where $S_{V,V}$ is an entry of the
$S$-matrix associated to $V.$ The entry $S_{W,V^G}$ of the
$S$-matrix associated to $V^G$ for any irreducible $V^G$-module $W$
is determined in this paper. It turns out that
$S_{V^G,V^G}=\frac{1}{|G|}S_{V,V}.$ This leads to the well known
formula $\glob(V^G)=|G|^2\glob(V)$ which was known using the
category theory \cite{H}, \cite{DMNO},\cite{KO}, \cite{ENO} (see
Theorem 2.9 of \cite{ADJR}). The same formula also appears in the
setting of conformal nets in \cite{X} and \cite{KLX}. Moreover, we
discover that $|G|\glob(V)$ is a sum of squares of the quantum
dimensions of all irreducible $g$-twisted modules where $g$ runs
through $G.$ This suggests a stronger result that for any
automorphism $g$ of $V$ of finite order, $ \glob(V)$ is a sum of
squares of the quantum dimensions of the irreducible $g$-twisted
$V$-modules.  A proof of this result  involves with the fusion rules and
the $S$-matrix of $V^G$ with $G$ being the cyclic group generated by
$g.$

The paper is organized as follows. In Section 2, we define twisted modules, $g$-rationality following \cite{DLM4}.
We also review some results used in this paper for $g$-rational vertex operator algebras and the modular invariance of trace functions from \cite{DLM7}. The transformation of these functions by the $S=\left(\begin{array}{cc} 0 &-1\\ 1 &0\end{array}\right)$ plays an essential role in the computation of the quantum dimensions and the global dimensions. In Section 3, we first review the irreducible $V^G$-modules occurring in irreducible $g$-twisted modules from \cite{DLM2}, \cite{DY} and \cite{MT}. We then establish that
these modules are complete by showing that $Z_{V^G}(v,\frac{-1}{\tau})$ (see the definition of  $Z_{V^G}(v,\frac{-1}{\tau})$ in Section 2) is a linear combination of $\tau^{\wt[v]}Z_{W}(v,\tau)$ where $W$ runs through the inequivalent irreducible $V^G$-modules occurring in irreducible $g$-twisted modules
for $g\in G.$ Section 4 is devoted to the study of quantum dimensions and global dimensions. First,
the quantum dimension of any irreducible $g$-twisted module $M$ is computed by using the $S$-matrix.
The quantum dimension of any irreducible $V^G$-submodule of $M$  is obtained using the quantum dimension of $M.$ This leads to a global dimension formula of $V$ in terms of quantum dimensions of all irreducible
twisted modules. We explain these results using the Heisenberg vertex operator algebras, the rank one lattice vertex operator algebras and the tensor product of vertex operator algebras. In Section 5  we show that
for any automorphism $g$ of $V$ of finite order, $ \glob(V)$ is a sum of squares of the quantum dimensions of the irreducible $g$-twisted $V$-modules. As an application we prove that if $V$ is holomorphic and $G$ is a cyclic group, then every irreducible $V^G$-module is a simple current. We also show the $V^G$-module
category is equivalent to the $D(G)$-module category as tensor categories where $D(G)$ is the Drinfeld double associated to $G.$

\section{Preliminary}

The various notions of twisted modules for a \voa following \cite{DLM4} are reviewed in this section. The concepts such as  rationality, regularity, and $C_2$-cofiniteness from \cite{Z} and \cite{DLM3} are discussed.
The modular invariance property of the trace functions in orbifold theory  from \cite{Z} and \cite{DLM7} are also given.

\subsection{Basics}

Let $V$ be a vertex operator algebra and $g$ an automorphism of $V$ of finite order $T$. Then  $V$ is a direct sum of eigenspaces of $g:$
\begin{equation*}\label{g2.1}
V=\bigoplus_{r\in \Z/T\Z}V^r,
\end{equation*}
where $V^r=\{v\in V|gv=e^{-2\pi ir/T}v\}$.
We use $r$ to denote both
an integer between $0$ and $T-1$ and its residue class \m $T$ in this
situation.

\begin{de} \label{weak}
A {\em weak $g$-twisted $V$-module} $M$ is a vector space equipped
with a linear map
\begin{equation*}
\begin{split}
Y_M: V&\to (\End\,M)[[z^{1/T},z^{-1/T}]]\\
v&\mapsto\displaystyle{ Y_M(v,z)=\sum_{n\in\frac{1}{T}\Z}v_nz^{-n-1}\ \ \ (v_n\in
\End\,M)},
\end{split}
\end{equation*}
which satisfies the following:  for all $0\leq r\leq T-1,$ $u\in V^r$, $v\in V,$
$w\in M$,
\begin{eqnarray*}
& &Y_M(u,z)=\sum_{n\in \frac{r}{T}+\Z}u_nz^{-n-1} \label{1/2},\\
& &u_lw=0~~~
\mbox{for}~~~ l\gg 0,\label{vlw0}\\
& &Y_M({\mathbf 1},z)=Id_M,\label{vacuum}
\end{eqnarray*}
 \begin{equation*}\label{jacobi}
\begin{array}{c}
\displaystyle{z^{-1}_0\delta\left(\frac{z_1-z_2}{z_0}\right)
Y_M(u,z_1)Y_M(v,z_2)-z^{-1}_0\delta\left(\frac{z_2-z_1}{-z_0}\right)
Y_M(v,z_2)Y_M(u,z_1)}\\
\displaystyle{=z_2^{-1}\left(\frac{z_1-z_0}{z_2}\right)^{-r/T}
\delta\left(\frac{z_1-z_0}{z_2}\right)
Y_M(Y(u,z_0)v,z_2)},
\end{array}
\end{equation*}
where $\delta(z)=\sum_{n\in\Z}z^n$ and
all binomial expressions (here and below) are to be expanded in nonnegative
integral powers of the second variable.
\end{de}

\begin{de}\label{ordinary}
A $g$-{\em twisted $V$-module} is
a $\C$-graded weak $g$-twisted $V$-module $M:$
\begin{equation*}
M=\bigoplus_{\lambda \in{\C}}M_{\lambda}
\end{equation*}
where $M_{\l}=\{w\in M|L(0)w=\l w\}$ and $L(0)$ is the component operator of $Y(\omega,z)=\sum_{n\in \Z}L(n)z^{-n-2}.$ We also require that
$\dim M_{\l}$ is finite and for fixed $\l,$ $M_{\frac{n}{T}+\l}=0$
for all small enough integers $n.$

If $w\in M_{\l}$ we refer to $\l$ as the {\em weight} of
$w$ and write $\l=\wt w.$
\end{de}

We use $\Z_+$ to denote the set of nonnegative integers.
\begin{de}\label{admissible}
 An {\em admissible} $g$-twisted $V$-module
is a  $\frac1T{\Z}_{+}$-graded weak $g$-twisted $V$-module $M:$
\begin{equation*}
M=\bigoplus_{n\in\frac{1}{T}\Z_+}M(n)
\end{equation*}
satisfying
\begin{equation*}
v_mM(n)\subseteq M(n+\wt v-m-1)
\end{equation*}
for homogeneous $v\in V,$ $m,n\in \frac{1}{T}{\Z}.$
\ed

If $g=Id_V$  we have the notions of  weak, ordinary and admissible $V$-modules \cite{DLM3}.

If $M=\bigoplus_{n\in \frac{1}{T}\Z_+}M(n)$
is an admissible $g$-twisted $V$-module, the contragredient module $M'$
is defined as follows:
\begin{equation*}
M'=\bigoplus_{n\in \frac{1}{T}\Z_+}M(n)^{*},
\end{equation*}
where $M(n)^*=\Hom_{\C}(M(n),\C).$ The vertex operator
$Y_{M'}(a,z)$ is defined for $a\in V$ via
\begin{eqnarray*}
\langle Y_{M'}(a,z)f,u\rangle= \langle f,Y_M(e^{zL(1)}(-z^{-2})^{L(0)}a,z^{-1})u\rangle,
\end{eqnarray*}
where $\langle f,w\rangle=f(w)$ is the natural paring $M'\times M\to \C.$
It follows from \cite{FHL} and \cite{X} that $(M',Y_{M'})$ is an admissible $g^{-1}$-twisted $V$-module.
We can also define the contragredient module $M'$ for a $g$-twisted $V$-module $M.$ In this case,
$M'$ is a $g^{-1}$-twisted $V$-module. Moreover, $M$ is irreducible if and only if $M'$ is irreducible.

\begin{de}
A \voa $V$ is called $g$-rational, if the  admissible $g$-twisted module category is semisimple. $V$ is called rational if $V$ is $1$-rational.
\end{de}

There is another important concept called $C_2$-cofiniteness \cite{Z}.
\begin{de}
We say that a \voa $V$ is $C_2$-cofinite if $V/C_2(V)$ is finite dimensional, where $C_2(V)=\langle v_{-2}u|v,u\in V\rangle.$
\end{de}

The following results about $g$-rational \voas \  are well-known \cite{DLM4}, \cite{DLM7}.
\begin{thm}\label{grational}
If $V$ is $g$-rational, then

(1) Any irreducible admissible $g$-twisted $V$-module $M$ is a $g$-twisted $V$-module. Moreover, there exists a number $\l \in \mathbb{C}$ such that  $M=\oplus_{n\in \frac{1}{T}\mathbb{Z_+}}M_{\l +n}$ where $M_{\lambda}\neq 0.$ The $\l$ is called the conformal weight of $M;$

(2) There are only finitely many irreducible admissible  $g$-twisted $V$-modules up to isomorphism.

(3) If $V$ is also $C_2$-cofinite and $g^i$-rational for all $i\geq 0$ then the central charge $c$ and the conformal weight $\l$ of any irreducible $g$-twisted $V$-module $M$ are rational numbers.
\end{thm}

\begin{de}
A \voa $V$ is called regular if every weak $V$-module is a direct sum of irreducible $V$-modules.
\end{de}

A \voa $V=\oplus_{n\in \Z}V_n$  is said to be of CFT type if $V_n=0$ for negative
$n$ and $V_0=\C {\bf 1}.$
It is proved in \cite{ABD} that if  $V$ is of CFT type, then regularity is equivalent to rationality and $C_2$-cofiniteness. Also $V$ is regular if and only if the weak module category is semisimple \cite{DYu}.

\subsection{Modular Invariance}

Let $V$ be a vertex operator algebra, $g$ an automorphism of $V$ of order $T$ and $M=\oplus_{n\in\frac{1}{T}\Z_+} M_{\l+n}$ a $g$-twisted $V$-module.

For any homogeneous element $v\in V$ we define a trace function associated to $v$ as follows:
\begin{equation*}
Z_M(v,q)=\tr_{M}o(v)q^{L(0)-c/24}=q^{\lambda-c/24}\sum_{n\in\frac{1}{T}\Z_+} \tr_{M_{\l+n}}o(v)q^n
\end{equation*}
where $o(v)=v(\wt v-1)$ is the degree zero operator of $v$. This is a formal power series in variable $q.$
It is proved in \cite{Z} and \cite{DLM7} that
$Z_M(v,q)$ converges to a holomorphic function on the domain $|q|<1$ if $V$ is $C_2$-cofinite.
We also use $Z_M(v,\tau)$ to denote the holomorphic function $Z_M(v,q).$
Here and below $\tau$ is in the complex upper half-plane $\H$ and $q=e^{2\pi i\tau}.$

Note that if $v=\1$ is the vacuum then $Z_M(\1,q)$ is the formal character of $M.$ We simply denote
$Z_M(\1,q)$ and $Z_M(1,\tau)$ by $\chi_M(q)$ and $\chi_M(\tau),$ respectively. $\chi_M(q)$ is called the character of $M.$

To proceed further, we need the action of $\Aut(V)$ on twisted modules \cite{DLM7}. Let $g, h$ be two automorphisms of $V$ with $g$ of finite order. If $(M, Y_M)$ is a weak $g$-twisted $V$-module, there is a weak $h^{-1}gh$-twisted  $V$-module $(M\circ h, Y_{M\circ h})$ where $M\circ h\cong M$ as vector spaces and
\begin{equation*}
Y_{M\circ h}(v,z)=Y_M(hv,z)
\end{equation*}
for $v\in V.$
This defines a left action of $\Aut(V)$ on weak twisted $V$-modules and on isomorphism
classes of weak twisted $V$-modules. Symbolically, we write
\begin{equation*}
(M,Y_M)\circ h=(M\circ h,Y_{M\circ h})= M\circ h,
\end{equation*}
where we sometimes abuse notation slightly by identifying $(M, Y_M)$ with the isomorphism
class that it defines.

If $g,h$ commute, $h$ clearly acts on the $g$-twisted modules.
 Denote by $\mathscr{M}(g)$ the equivalence classes of irreducible $g$-twisted $V$-modules and set $\mathscr{M}(g,h)=\{M \in \mathscr{M}(g)| h\circ M\cong M\}.$ Note from Theorem \ref{grational} that
 if $V$ is $g$-rational, both $\mathscr{M}(g)$ and $\mathscr{M}(g,h)$ are finite sets.
 For any $M\in \mathscr{M}(g,h),$ there is a $g$-twisted $V$-module isomorphism
\begin{equation*}
\varphi(h) : M\circ h\to M.
\end{equation*}
The linear map $\varphi(h)$ is unique up to a nonzero scalar. If $h=1$ we simply take $\varphi(1)=1.$
For $v\in V$ we set
\begin{equation*}
Z_M(v, (g,h),\tau)=\tr_{_M}o(v)\varphi(h) q^{L(0)-c/24}=q^{\lambda-c/24}\sum_{n\in\frac{1}{T}\Z_+}\tr_{_{M_{\l+n}}}o(v)\varphi(h)q^{n}
\end{equation*}
which is a holomorphic function on $\H$ \cite{DLM7}. Note that $Z_M(v, (g,h),\tau)$ is defined up to a nonzero scalar. It is evident that if $h=1$ then $Z_M(v,(g,1),\tau)=Z_M(v,\tau).$

In the rest of this  paper we assume the following:
\begin{enumerate}
\item[(V1)] $V=\oplus_{n\geq 0}V_n$ is a simple vertex operator algebra  of CFT type,
\item[(V2)] $G$ is a finite automorphism group of $V$ and $V^G$ is a vertex operator algebra of CFT type,
\item[(V3)] $V^G$ is $C_{2}$-cofinite and  rational,
\item[(V4)] The conformal weight of any  irreducible $g$ twisted $V$-module for $g\in G$ except
$V$ itself is positive.
\end{enumerate}

The following results are obtained in \cite{ABD},  \cite{ADJR} and \cite{HKL}.
\begin{lem}\label{general} Let $V$ and $G$ be as before. Then

(1)  $V$ is $C_2$-cofinite,

(2)  $V$ is $g$-rational for all $g\in G.$
\end{lem}

There is another vertex operator algebra $(V, Y[~], \1, \tilde{\omega})$ associated to $V$
in \cite{Z}.  Here $\tilde{\omega}=\omega-c/24$ and
$$Y[v,z]=Y(v,e^z-1)e^{z\cdot \wt v}=\sum_{n\in \Z}v[n]z^{n-1}$$
for homogeneous $v.$ We also write
$$Y[\tilde{\omega},z]=\sum_{n\in \Z}L[n]z^{-n-2}.$$
If $v\in V$ is homogeneous in the second vertex operator algebra, we denote its weight by $\wt [v].$

Let $P(G)$ denote the ordered commuting pairs in $G.$
For $(g,h)\in P(G)$ and $M\in \mathscr{M}(g,h),$ $Z_M(v,(g,h),\tau)$ is a function
on $V\times \H.$ Let $W$ be the vector space spanned by these functions. It is clear that
the dimensional $d$ of $W$ is equal to $\sum_{(g,h)\in P(G)}|\mathscr{M}(g,h)|.$
We now define an action of the modular group $\Gamma$ on $W$ such that
\begin{equation*}
Z_M|_\gamma(v,(g,h),\tau)=(c\tau+d)^{-{\rm wt}[v]}Z_M(v,(g,h),\gamma \tau),
\end{equation*}
where
\begin{equation}\label{e1.11}
\gamma: \tau\mapsto\frac{ a\tau + b}{c\tau+d},\ \ \ \gamma=\left(\begin{array}{cc}a & b\\ c & d\end{array}\right)\in\Gamma=SL(2,\Z).
\end{equation}
 We let $\gamma\in \Gamma$ act on the right of $P(G)$ via
$$(g, h)\gamma = (g^ah^c, g^bh^d ).$$

Using Lemma \ref{general} we have the following results \cite{DLM7}:
\begin{thm}\label{minvariance} (1) There is a representation $\rho: \Gamma\to GL(W)$
such that for $(g, h)\in P(G),$   $\gamma =\left(\begin{array}{cc}a & b\\ c & d\end{array}\right)\in \Gamma,$
and $M\in \mathscr{M}(g,h),$
$$
Z_{M}|_{\gamma}(v,(g,h),\tau)=\sum_{N\in \mathscr{M}(g^ah^c,{g^bh^d)}} \gamma_{M,N} Z_{N}(v,(g, h)\gamma, ~\tau)$$
where $\rho(\gamma)=(\gamma_{M,N}).$
That is,
$$Z_{M}(v,(g,h),\gamma\tau)=(c\tau+d)^{{\rm wt}[v]}\sum_{N\in \mathscr{M}(g^ah^c,{g^bh^d)}} \gamma_{M,N} Z_{N}(v,(g, h)\gamma, ~\tau).$$

(2) The cardinalities $|\mathscr{M}(g,h)|$ and $|\mathscr{M}(g^ah^c,g^bh^d)|$ are equal for any $(g,h)\in P(G)$ and
$\gamma\in \Gamma.$ In particular, the number of inequivalent irreducible $g$-twisted $V$-modules is exactly the number of irreducible $V$-modules which are  $g$-stable.
\end{thm}

\begin{rem}\label{trans} There are two special cases most in the present paper:
If $\gamma=S=\left(\begin{array}{cc}0 & -1\\ 1 & 0\end{array}\right)$ we
 have:
\begin{equation}\label{S-tran1}
Z_{M}(v,-\frac{1}{\tau})=\tau^{\wt[v]}\sum_{N\in \mathscr{M}(1,g^{-1}) }S_{M,N} Z_{N}(v,(1,g^{-1}),\tau)
\end{equation}
for $M\in \mathscr{M}(g)$ and
\begin{equation}\label{S-tran2}
Z_{N}(v,(1,g),-\frac{1}{\tau})=\tau^{\wt[v]}\sum_{M\in \mathscr{M}(g)} S_{N,M} Z_{M}(v,\tau).
\end{equation}
\end{rem}
Since we are dealing with various vertex operator algebra, we use $\mathscr{M}_V$ for $\mathscr{M}(1).$
The matrix $S=(S_{M,N})_{M,N\in \mathscr{M}_V}$ is called $S$-matrix of $V$ and is independent of the choice of vector $v.$

\subsection{Fuison rules and Verlinde Fromula}

Let $V$ be as before and $M,N,W\in  \mathscr{M}_V.$
The fusion rule  $N_{M,N}^W$ is  $\dim I_V \left(
\begin{array}{c}
\ \ W \ \\
M \ \  N
\end{array}
\right),$ where
$I_V \left (
\begin{array}{c}
\ \ W \ \\
M \ \  N
\end{array}
 \right)$
 is the space of intertwining operators of type
 $\left(
\begin{array}{c}
\ \ W \ \\
M \ \  N
\end{array}
 \right).$

Since $V$ is rational there is a tensor product $\boxtimes$ of two $V$-modules \cite{HL1}-\cite{HL3} such that
if $M,N$ are irreducible then $M\boxtimes N=\sum_{W\in  \mathscr{M}_V}N_{M,N}^WW.$  The irreducible
$V$-module is called a simple current if $M\boxtimes N$ is irreducible again for any irreducible module $N.$

The following Verlinde formula \cite{V} was proved in \cite{H}.
\begin{thm}\label{Verlinde} Let $V$ be as before. Then

(1) $(S^{-1})_{M,N}=S_{M,N'}=S_{M',N},$ and $S_{M',N'}=S_{M,N}; $

(2) $S$ is symmetric and $S^2=(\delta_{M,N'});$

(3) $N_{M,N}^W=\sum_{U\in  \mathscr{M}_V}\frac{S_{N,U}S_{M,U}S^{-1}_{U,W}}{S_{V,U}};$

(4) The $S$-matrix diagonalizes the fusion matrix $F(M)=(N_{M,N}^W)_{N,W\in  \mathscr{M}_V}.$ More explicitly, $S^{-1}F(M)S=\diag(\frac{S_{M,N}}{S_{V,N}})_{N\in  \mathscr{M}_V}.$  In particular, $S_{V,N}\neq 0$ for any $N.$
\end{thm}

We also have \cite{DLN}:
\begin{prop} \label{pDLN}The $S$-matrix is unitary and $S_{V,M}=S_{M,V}$ is positive for any irreducible $V$-module $M.$
\end{prop}

A simple current $M$ has order $n>0$ if $n$ is the least positive integer such that $M^{\boxtimes n}=V.$  In fact, the simple currents form an abelain group with product $\boxtimes$ and $n$ is exactly the order of $M$ in this group.
\begin{lem}\label{simplefusion} If $M$ is a simple current of order $n$ then  $ \frac{S_{M,N}}{S_{V,N}}$ is a $n$-th root of unity
for any irreducible $V$-module $N.$ Moreover, $F(M^{\boxtimes r})=F(M)^r$ for any integer $r.$
\end{lem}
\begin{proof}  Consider the vector space  with a basis $\mathscr{M}_V.$ The $\boxtimes $ makes this space a module for the group of simple currents. The corresponding matrix of $M$ with respect to the basis
 $\mathscr{M}_V$ is exactly  $F(M).$ So the eigenvalues $\frac{S_{M,N}}{S_{V,N}}$ of $F(M)$ are the $n$-th root of unity.  The relation $F(M^{\boxtimes r})=F(M)^r$ is clear.
\end{proof}

\section{Classification of irreducible modules for $V^G$}
In this section we give a classification of irreducible modules of $V^G$ and show that any irreducible $V^G$ module occurs in an irreducible $g$-twisted $V$-module for some $g\in G.$

Let $M=(M,Y_M)$ be an irreducible $g$-twisted $V$-module.
We define a subgroup $G_M$ of $G$ consisting of $h\in G$ such that $M\circ h$ and $M$ are isomorphic.
As we mentioned in Section 2 there  is a projective
representation $h\mapsto \phi(h)$ of $G_M$ on $M$ such that
$$
\phi(h)Y_M(v,z)\phi(h)^{-1}=Y_M(hv,z)
$$
for $h\in G_M$ and $v\in V.$
Let $\a_M$  be the corresponding 2-cocycle in $C^2(G,\C^{\times}).$ Then $\phi(h)\phi(k)=\a_M(h,k)\phi(hk)$
for all $h,k\in G_M.$ We may  assume that $\alpha_M$ is unitary  \cite{C}. That is, there is a fixed positive integer $n$ such that  $\alpha_M(h,k)^n=1$ for all $h,k\in G_M.$  Let $\C^{\a_M}[G_M]=\oplus_{h\in G_M} \C\bar h$ be the twisted group algebra with product $\bar h\bar k=\alpha_M(h,k)\bar{hk}.$ It is well known that $\C^{\a_M}[G_M]$ is a semisimple associative algebra. It follows that $M$ is a $\C^{\a_M}[G_M]$-module.

 Note that $G_M$ is a subgroup of $C_G(g).$ We claim that $g$ lies in $G_M.$ From the definition, $M$ has a decomposition $M=\oplus_{n\in \frac{1}{T}\Z_+}M(n)$ and $M(0)\ne 0.$
 We define $\phi(g)$ acting on $M(n)$ as $e^{2\pi in}$ for all $n.$ It is easy to check that
$$
\phi(g)Y_M(v,z)\phi(g)^{-1}=Y_M(gv,z)
$$
for  $v\in V.$ This implies $g\in G_M.$ For $r=0,...,T-1$ let $M^r=\oplus_{n\in \frac{r}{T}+\Z}M(n).$ Then $M=\oplus_{r=0}^{T-1}M^r$
and each $M^r$ is an irreducible $V^{\< g\>}$-module. Clearly, $\phi(g)$ acts on each $M^r$ as constant $e^{2\pi ir/T}.$ This fact will be useful later.

Let $\Lambda_{G_M,\a_M}$ be the set of all irreducible characters $\lambda$
of  $\C^{\a_M}[G_M]$. Denote
the corresponding simple module by $W_{\l}.$
Using the fact  that $M$ is a semisimple
$\C^{\a_M}[G_M]$-module, we let $M^{\lambda}$ be the sum of simple
$\C^{\a_M}[G_M]$-submodules of $M$ isomorphic
to $W_{\l}.$ Then
$$M=\oplus_{\lambda\in \Lambda_{G_M,\a_M}}M^{\lambda}.$$
Moreover, $M^{\lambda}=W_{\lambda}\otimes M_{\l}$
where $M_{\l}=\hom_{\C^{\a_M}[G_M]}(W_{\l},M)$ is the multiplicity
of $W_{\l}$ in $M.$ As in \cite{DLM2}, we can,
in fact, realize $M_{\l}$ as a subspace of $M$ in the following way.  Let $w\in W_{\l}$
be a fixed nonzero vector. Then we can identify
$\hom_{\C^{\a_M}[G_M]}(W_{\l},M)$ with the subspace
$$\{f(w) |f\in \hom_{\C^{\a_M}[G_M]}(W_{\l},M)\}$$
 of $M^{\l}.$ This gives a decomposition
\begin{equation}\label{decom}
M=\oplus_{\lambda\in \Lambda_{G_M,\a_M}}W_{\l}\otimes M_{\lambda}
\end{equation}
 and each $M_{\l}$ is a module for vertex operator subalgebra $V^{G_M}$-module.

 \begin{thm}\label{mthm1} With the same notation as above we have:

(1) $M^{\l}$ is nonzero for any $\l\in \Lambda_{G_M,\a_M}.$

(2) Each $M_{\l}$ is an irreducible $V^{G_M}$-module.

(3) $M_{\l}$ and $M_{\gamma}$ are equivalent $V^{G_M}$-module if and only
if $\l=\gamma.$
\end{thm}

The proof of this theorem is similar to that of Theorem 5.4 of \cite{DY} using the twisted associative algebras
$A_{g,n}(V)$ \cite{DLM5}, \cite{DLM6}.  We refer the reader to \cite{DY} for details.

Recall that the group $G$ acts on set ${\cal S}=\cup_{g\in G}\mathscr{M}(g)$ and $ M\circ h$ and $M$ are isomorphic
$V^G$-modules for any $h\in G$ and $M\in {\cal S}.$ It is clear that the cardinality of the $G$-orbit $M\circ G$ of $M$ is
$[G:G_M].$

The following result comes from \cite{MT} (also see \cite{DLM2}, \cite{DY})  which is a generalization of Theorem \ref{mthm1}.

\begin{thm}\label{MT} Let $g,h\in G,$ $M$ an irreducible $g$-twisted $V$-module, $N$ an irreducible $h$-twisted $V$-module. Also assume that $M,N$ are not in the same orbit of ${\cal S}$ under the action of $G.$ Then

(1) Each $M_{\l}$ for $\lambda\in \Lambda_{G_M,\a_M}$ is an irreducible $V^G$-module.

(2) For any $\lambda\in \Lambda_{G_M,\a_M}$ and $\mu \in \Lambda_{G_N,\a_N},$ the irreducible $V^G$-modules
$M_{\l}$ and $N_{\mu}$ are inequivalent.
\end{thm}

This result gives a complete classification of irreducible $V^G$-modules occurring in the irreducible $g$-twisted $V$-modules for $g\in G.$ We now use the modular invariance to prove that these irreducible $V^G$-modules are  the all irreducible $V^G$-modules.
\begin{thm}\label{MTH1} Any irreducible $V^G$-module is isomorphic to  $M_{\l}$ for some
irreducible $g$-twisted $V$-module $M$ and some $\lambda\in \Lambda_{G_M,\a_M}.$
\end{thm}

\begin{proof} Recall from Section 2.2 about the $S$-matrix for vertex operator algebra $V^G.$ We know that
$$Z_{V^G}(v,-\frac{1}{\tau})=\tau^{\wt[v]}\sum_{W\in  \mathscr{M}_{V^G}}S_{V^G,W}Z_W(v,\tau)$$
for $v\in V^G$ and  $Z_W(v,\tau)$ are linearly independent \cite{Z}.
According to \cite{H}, $S_{V^G,W}\ne 0$  for all $W.$ So it is enough to show
that $\tau^{-\wt[v]}Z_{V^G}(v,-\frac{1}{\tau})$ is a linear combination of $Z_W(v,\tau)$ for $W$ occurring in irreducible
twisted $V$-modules.

From the definition of $Z_{V^G}(v,\tau)$ we know that for $v\in V^G$
$$Z_{V^G}(v,\tau)=\frac{1}{|G|}\sum_{g\in G}Z_V(v,(1,g),\tau).$$
Using (\ref{S-tran2}) and the orthogonality property of the irreducible characters of $G$
gives
$$Z_{V^G}(v,-\frac{1}{\tau})=\frac{\tau^{\wt[v]}}{|G|}\sum_{g\in G}\sum_{M\in \mathscr{M}(g)} S_{V,M} Z_{M}(v,\tau).$$
By Theorems \ref{mthm1} and \ref{MT},
$$Z_M(v,\tau)=\sum_{\lambda\in \Lambda_{G_M,\alpha_M}}\dim W_{\l}Z_{M_{\l}}(v,\tau)$$
and the proof is complete.
 \end{proof}

\begin{rem}In the proof of Theorem \ref{MTH1} we did not use the assumption (V4), but we did need to assume that $V^G$ is self-dual.
\end{rem}

The following result is an immediate consequence of Theorem \ref{MT} and assumption (V4).
\begin{coro} The weight of every irreducible $V^G$-module is positive except $V^G$ itself.
\end{coro}

This result allows us to use various  results on the quantum dimensions  for vertex operator algebra
$V^G$ \cite{DJX}.

We now discuss the modularity of $Z_M(v,(g,h),\tau).$ Recall from \cite{DLN} that for any $v\in V,$
and $\l\in\Lambda_{G_M,\a_M}, $ $Z_{M_\l}(v,\tau)$ is a modular form of weight $\wt[v]$ over a congruence
subgroup $A$ of $\Gamma.$
\begin{prop}Let $V,$ $G,$ $g,$ be as before. Then for any $v\in V$ and $h\in G_M,$ $Z_M(v,(g,h),\tau)$
is a modular form of weight $\wt[v]$ over $A.$ In particular, $\chi_M(\tau)$ is a modular function on $A.$
\end{prop}
\begin{proof} Using the decomposition we see that
$$Z_M(v,(g,h),\tau)=\sum_{\lambda\in\Lambda_{G_M,\a_M}}\lambda(h)Z_{M_\l}(v,\tau).$$
The result follows.
\end{proof}

\section{Quantum dimensions}
In this section we compute the quantum dimensions of the irreducible $g$-twisted $V$-modules and irreducible $V^G$-modules. We also obtain  relations between the global dimensions of $V$ and $V^G,$
and the global dimension of $V$ and the quantum dimensions of irreducible $g$-twisted $V$-modules.

\subsection{Definition}
Let $V$ and $G$ be as before. Let $g\in G$ and $M$ a $g$-twisted $V$-module. Then $M$ is a finite sum of irreducible $g$-twisted $V$-modules. In particular, each homogeneous subspace of $M$ is finite dimensional.

From the discussion before we know $\chi_V(\tau)$ and $\chi_M(\tau)$ are holomorphic functions on $\H.$ The quantum dimension of $M$ over $V$ is defined to be
$$\qdim_{V}M=\lim_{y\to 0}\frac{\chi_M(iy)}{\chi_V(iy)}$$
where $y$ is real and positive. In the case $g=1$ this is exactly the definition of the quantum dimension
of a $V$-module given in \cite{DJX}. Alternatively, we can define quantum dimension as
$$
\qdim_VM=\lim_{q\to 1^-}\frac{\ch_qM}{\ch_q V}$$
using the relation $q=e^{2\pi i\tau}.$

Note that it is not obvious that the $\qdim_{V}M$ exists from the definition. But for any $V$-module $M,$
$\qdim_{V}M$ always exists and is greater than or equal to 1 \cite{DJX}. We will prove later that for any $g$
and any $g$-twisted $V$-module $M,$ $\qdim_{V}M$  always exists. It is clear that $\qdim_{V}M$ is nonnegative. Also, $\qdim_V(h\circ M)=\qdim_VM$
for any $h\in G$ and $\qdim_V M=\qdim_V M'.$

We now see two examples. First consider the Heisenberg \voa $M(1)$  constructed from the vector space $H$ of dimension $d$ and with a nondegenerate symmetric bilinear form. The quantum dimension of any irreducible $M(1)$-module
has been computed in \cite{DJX}. We now consider the automorphism $\theta$ of $M(1)$ induced from the $-1$
linear map on $H.$  Then $M(1)$ has a unique irreducible $\theta$-twisted module $M(1)(\theta)$ (see \cite{FLM2} for the details). The following character formula for $M(1)(\theta)$ is well known:
$$
\chi_{M(1)}(q)=\frac{1}{\eta(q)^d},\ \chi_{M(1)(\theta)}(q)=\frac{\eta(q)^d}{\eta(q^{1/2})^d}
$$
where
$$\eta(q)=q^{\frac{1}{24}}\prod _{n\geq 1}(1-q^n).$$
Using the well known transformation formula
$$\eta(-1/\tau)=(-i\tau)^{1/2}\eta(\tau)$$
gives
\begin{equation*}
\begin{split}
\qdim_{M(1)} M(1)(\theta)=&\lim_{y\to 0}\frac{\eta(iy)^{2d}}{\eta(\frac{iy}{2})^d}\\
=&\lim_{y\to \infty}\frac{\eta(\frac{-1}{iy})^{2d}}{\eta(\frac{-1}{iy2})^d}\\
=&\lim_{y\to \infty}\frac{y^d\eta(iy)^{2d}}{(2y)^{\frac{d}{2}}\eta(2iy)^d}\\
=&\infty.
\end{split}
\end{equation*}
Note that both $M(1)$  and the $\theta$-invariants
$M(1)^+$ are not rational and infinitely many irreducible modules.  Although the results later on quantum dimensions do not apply to $M(1)$ and $M(1)^+.$ But they may explain why $\qdim_{M(1)} M(1)(\theta)$ is infinity.

Our second example is the lattice vertex operator algebra $V_L$ where $L=\Z\alpha$ is an even positive definite lattice with $(\alpha,\alpha)=2k$ for $k\geq 1.$ We know from \cite{D1} and \cite{DLM3} that $V_L$ is rational
and irreducible modules are $V_{L+\frac{r}{2k}\alpha}$ for $r=0,...,2k-1.$ Again the $-1$ isometry of $L$ induces an automorphism of $V_L.$ In this case $V_L$ has exactly two irreducible $\theta$-twisted modules
$V_L^{T_i}=M(1)(\theta)\otimes T_i$ for $i=0,1$ where $T_i=\C$ are 1-dimensional $\Z\alpha$-module such that $\alpha$ acts as $(-1)^i$ \cite{FLM2}, \cite{D2}.

We know again from \cite{DJX} that each irreducible $V_{L+\frac{r}{2k}\alpha}$ has quantum dimension 1. The character of $V_L$ is given by
$$\chi_{V_L}(q)=\frac{\theta_L(q)}{\eta(q)}$$
where the theta function is defined as
$$\theta_{L}(q)=\sum_{n\in \Z}q^{(n\alpha,n\alpha)/2}=\sum_{n\in \Z}q^{n^2k}$$
with the transformation law
$$\theta_{L}(-1/\tau)=\frac{1}{\sqrt{2k}}(-i\tau)^{1/2}\theta_{L^{\circ}}(\tau)$$
\cite{S}.
Notice that $L^{\circ}=\frac{1}{2k}\Z\alpha$ is the dual lattice of $L.$
Thus
\begin{equation*}
\begin{split}
\qdim_{V_L}V_L^{T_i}=&\lim_{y\to 0} \frac{\eta(iy)^2}{\theta_L(iy)\eta(\frac{iy}{2})}\\
=&\lim_{y\to \infty} \frac{\eta(\frac{-1}{iy})^2}{\theta_L(\frac{-1}{iy})\eta(\frac{-1}{2iy})}\\
=&\lim_{y\to \infty}\sqrt{k}\frac{\eta(iy)^2}{\theta_{L^{\circ}}(iy)\eta(2iy)}\\
=&\sqrt{k}.
\end{split}
\end{equation*}

In both examples we use the transformation law of certain modular forms to compute the quantum dimension. But the quantum dimension $\qdim_{V_L}V_L^{T_i}$ will also follow a general result on rational vertex operator algebra.

\subsection{Quantum dimensions of twisted modules}

The existence of the quantum dimension for a $g$-twisted $V$-module is established in this subsection. In fact， we give a explicit formula of the quantum dimension in terms of the $S$-matrix.

\begin{prop}\label{tqdim}
Let $V$ and $G$ be as before, and $M$ a $g$-twisted $V$-module for some $g\in G.$ Then $\qdim_VM$ exists.
If $M$ is irreducible then $\qdim_VM=\frac{S_{M,V}}{S_{V,V}}.$
\end{prop}
\begin{proof}
The proof is similar to that of Lemma 4.2 of \cite{DJX}. Since $M$ is a direct sum of finitely many irreducible $g$-twisted $V$-modules, we can assume that $M$ is irreducible. Using (\ref{S-tran1}) with $v=\1$
we have
\begin{equation}\label{qdim S}
\begin{split}
\qdim_{V} M=&\lim_{y\to 0}\frac{\chi_{M}(iy)}{\chi_V(iy)}\\
=&\lim_{ y\to \infty}\frac{\chi_{M}(-\frac{1}{iy})}{\chi_{V}(-\frac{1}{iy})}\\
=&\lim_{y \to \infty}\frac{\sum_{N\in \mathscr{M}(1,g^{-1}) }S_{M,N} Z_{N}(\1,(1,g^{-1}),iy)}{\sum_{N\in \mathscr{M}_V} S_{V,N}\chi_{N}(iy)}\\
=&\frac{S_{M,V}}{S_{V,V}}
\end{split}
\end{equation}
where we have used the fact that $lim_{y\to \infty}\frac{Z_{N}(\1,(1,g^{-1}),iy)}{\chi_{V}(iy)}=\delta_{V,N}$
as the conformal weight of $N$ is positive if $N\ne V.$
\end{proof}

\begin{rem}\label{positivity} Although $S_{V,V}>0$ \cite{H}, \cite{DLN}, but it is not clear at this point that $S_{M,V}$ is positive. We will show in the next subsection that $S_{M,V}$ is always positive. In other words,
$\qdim_VM$ is always positive.
\end{rem}

\subsection{Quantum dimensions of $V^G$-modules}

In this subsection we compute the quantum dimensions of irreducible $V^G$-modules

We first need a result on the projective representations of a finite group $K.$ Let $\alpha$ be a unitary
2-cocycle of $K.$ That is, $\alpha(a,b)\in \<\kappa\>$ for all $a,b\in K$ where $\kappa=e^{2\pi i/n}.$
 Let $\hat K=K\times \<\kappa\>.$
 Then $\hat K$ is a finite group with product $(a, \kappa^s)(b,\kappa^t)=(ab,\alpha(a,b)\kappa^{s+t})$
 for $a,b\in K$ and $s,t\in\Z.$   Then $\hat K$ is a central extension of $K$ and the twisted group algebra $\C^{\alpha}[K]$ can be regraded as the quotient of $\C[\hat K]$ by identifying the abstract group element $\kappa$ with $e^{2\pi i/n}.$ Let $\chi_1, \chi_2$ be two irreducible characters of $\C^{\alpha}[K].$
 \begin{lem}\label{tga}
 Let $\chi_1, \chi_2$ be two irreducible characters of $\C^{\alpha}[K].$ The following orthogonal relation
 holds:
 $$\frac{1}{|K|}\sum_{a\in K}\chi_1(a)\overline{\chi_2(a)}=\delta_{\chi_1,\chi_2}.$$
 \end{lem}
 \begin{proof} We can regard $\chi_1, \chi_2$ as irreducible characters of $\hat K$ in an obvious way. Then
 $$\frac{1}{|\hat K|}\sum_{a\in K, 0\leq s\leq n-1 }\chi_1(a)\chi_1(\kappa)^s\overline{\chi_2(a) k^{s}}=\delta_{\chi_1,\chi_2}.$$
 The result follows immediately by noting that $|\hat K|=n|K|.$
\end{proof}
We are now in a position to compute the quantum dimensions of irreducible $V^G$-modules.
\begin{thm}\label{MTH} Let $V,G, g$ be as before, $M$ an irreducible $g$-twisted $V$-module and $\lambda\in \Lambda_{G_M,\a_M}.$ We have
\begin{equation}S_{M_\l,V^G}=\frac{\dim W_{\l}}{|G_M|}S_{M,V},
\end{equation}
\begin{equation}\label{eqq1}
\qdim_{V^G}M_{\l}=[G:G_M]\dim W_{\l}\qdim_{V}M.
\end{equation}
Moreover, $\qdim_{V}M$ takes values in $\{2\cos\frac{\pi}{n}|n\geq 3\}\cup [2,\infty)$ and $S_{M,V}$ is positive.
\end{thm}
\begin{proof}

 Using Lemma \ref{tga} we see that
 $$Z_{M_\l}(v, \tau)=\frac{1}{|G_M|}\sum_{h\in G_M}Z_M(v, (g,h),\tau)\overline{\lambda(h)}$$
  for $v\in V^G.$ By Theorem \ref{minvariance} we have
\begin{equation}\label{smatrix}
Z_{M_\l}(v, -1/\tau)=\frac{1}{|G_M|}\tau^{\wt[v]}\sum_{h\in G_M}\sum_{N\in \mathscr{M}(h,g^{-1})}S_{M,N}Z_N(v,(h,g^{-1}),\tau)\overline{\lambda(h)}.
\end{equation}
On the other hand,
$$Z_{M_\l}(v, -1/\tau)=\tau^{\wt[v]}\sum_{W\in \mathscr{M}_{V^G}}S_{M_\l,W}Z_W(v,\tau).$$
So we need to look for the coefficient of $Z_{V^G}(v,\tau)$ in the right hand side
of equation (\ref{smatrix}).

From the decomposition
$$N=\oplus_{\mu\in\Lambda_{G_N,\alpha_N}}W_{\mu}\otimes N_{\mu}$$
and Theorems \ref{mthm1}, \ref{MT} we know that $Z_N(v,(h,g^{-1}),\tau)=\sum_{\mu\in\Lambda_{G_N,\alpha_N}}\mu(g^{-1})Z_{N_{\mu}}(v,\tau)$ is a linear combination.
So if $N\ne V,$  $Z_N(v,(h,g^{-1}),\tau)$ does not contribute
$Z_{V^G}(v,\tau).$  If $N=V$ then $h=1$ and $G_V=G,$ and $\Lambda_{G_V,\alpha_V}$ is the set of irreducible characters. The coefficient of $Z_{V^G}(v,\tau)$ in  $Z_V(v,(1,g^{-1}),\tau)$ is 1. As a result we see the coefficient of $Z_{V^G}(v,\tau)$ in $\tau^{-\wt[v]}Z_{M_\l}(v, -1/\tau)$ is $\frac{\dim W_\l}{|G_M|}S_{M,V}.$

In the case that  $M=V$
 and $M_{\l}=V^G$ we see that $S_{V^G,V^G}= \frac{1}{|G|}S_{V,V}.$
It follows from Proposition \ref{tqdim} (also see the proof of Lemma 4.2 of \cite{DJX}) that
$$\qdim_{V^G}M_{\l}=\frac{S_{M_\l,V^G}}{S_{V^G,V^G}}=\frac{|G|}{|G_M|}\dim_{W_\l}\frac{S_{M,V}}{S_{V,V}}
=[G:G_M]\dim{W_\l}\qdim_VM.$$

To prove that the inequality  $\qdim_{V}M$ lies in $\{2\cos\frac{\pi}{n}|n\geq 3\}\cup [2,\infty)$  we take $G$ to be the cyclic group generated by $g.$
We have mentioned already that $G_M=G$ in this case and $M$ is a $G$-module. Since $G$ is abelian,
any irreducible module is one-dimensional.  Using  (\ref{eqq1}) and fact that $\qdim_{V^G}M_{\l}$ belongs to $\{2\cos\frac{\pi}{n}|n\geq 3\}\cup [2,\infty)$ \cite{DJX}  concludes that  $\qdim_{V}M$ lies in $\{2\cos\frac{\pi}{n}|n\geq 3\}\cup [2,\infty).$ The positivity of $S_{M,V}$
follows from Proposition \ref{tqdim},  the positivity of $S_{V,V}$ (see Remark \ref{positivity}) and $\qdim_{V}M.$
The proof is complete.
\end{proof}

We remark that equation (\ref{eqq1}) also appears in the conformal
net setting in Th. 4.5 of \cite{KLX}.

The following corollary is immediate by noting that $|G_M|=\sum_{\lambda\in\Lambda_{G_M,\alpha_M}}(\dim W_\l)^2.$
\begin{coro} We have
$$\qdim_{V^G}M=|G|\qdim_VM.$$
for any irreducible $g$-twisted $V$-module $M.$
\end{coro}

Theorem \ref{MTH} has been obtained previously in \cite{DJX} with $M=V.$ That is, $\qdim_{V^G} V_{\l}=\dim W_{\l}.$
We next explain the appearance of $\frac{|G|}{|G_M|}$ in the quantum dimension $\qdim_{V^G}M_{\l}$ and its relation with a Schur-Weyl duality \cite{DY}, \cite{MT}.

Recall ${\cal S}$ is the set of equivalent classes of irreducible $g$-twisted modules for $g\in G.$ We need to recall the construction of the finite dimensional semisimple associative algebra $A_\a(G,\S)$ defined in \cite{DY}. As we have mentioned already that $G$ acts on $\S$ by sending $(M,h)$ to $M\circ h$ for $h\in G$ and $M\in \S.$ Let $M\in {\S}$ and $x\in G$. Then
there exists $N\in {\S}$ such that
$N\cong M\circ x.$ That is,
there is a linear map
$\p_N(x):N\rightarrow M$ satisfying the condition:
$\p_N(x)Y_N(v,z)\p_N(x)^{-1}=Y_M(xv,z).$
By simplicity of $N$, there exists $\a_N(y,x)\in \C^*$
such that $\p_M(y)\p_N(x)=\a_N(y,x)\p_N(yx).$
Moreover, for $x,y,z\in G$ we have
$$\a_N(z,yx)\a_N(y,x)=\a_M(z,y)\a_N(zy,x).$$

Set $\C{\S}=\bigoplus_{M\in {\S}}\C e(M)$
define $e(M)e(N)=\d_{M,N}e(M)$. Then $\C{\S}$ is an associative algebra.
Let
$$U(\C{\cal S})
=\{\sum_{M\in{\S}}c_Me(M)|c_{M}\in \C^*\}$$
be the units of  $\C{\S}.$
Then
$$\a(h,k)=\sum_{M\in {\S}}\a_{M}(h,k)e(M)$$
lies in $U(\C{\cal S}).$
It is easy to check that
$$\a(hk,l)\a(h,k)^l=\a(h,kl)\a(k,l)$$
where
$$\a(h,k)^l=\sum_{M\in {\S}}\a_{M}(h,k)e(M\circ l).$$
So
$\a\in Z^2(G,U(\C\cal{S}))$ is a 2-cocycle.

The associative algebra
$A_{\a}(G,{\S})$ is defined as $\C[G]\otimes \C\S$
with multiplication
$$ a\ot e(M)\cdot b\ot e(N)=\a_{N}(a,b)ab\ot e(M\cdot b)e(N)$$
for $a,b\in G$ and $M,N\in\S.$ Then $A_{\a}(G,{\S})$ is a semisimple associative algebra.
Moreover, $\oplus_{N\in \S}N$ is an $A_{\a}(G,{\S})$-module with
the action: for $M,N\in {\S}$ and $w\in N$ we set
$$a\ot e(M)\cdot w= \d_{M,N}\p_{M}(a)w$$
where $\p_{N}(a):N\rightarrow N\cdot a^{-1}$ \cite{DY}.
It is clear that the action of $A_{\a}(G,{\S})$
and $V^G$ commute.
For $M\in {\S}$ we let ${\Or}_M=\{N\cdot x|x\in G\}$
be the orbit of $N$ under $G.$ It is clear that $\oplus_{N\in \Or_M}N$ is an $A_{\a}(G,{\S})$-submodule
of $\oplus_{N\in \S}N.$

Let $M\in \S$ and set $D(M)=\<a\ot e(M)|a\in G\>$ and $S(M)=\<a\ot e(M)|a\in G_M\>.$ Then both $D(M)$ and $S(M)$ are subalgebras of $A_\a(G,\S)$ and $S(M)$ is isomorphic to $\C^{\a_M}[G_M].$ Moreover, each $\Ind_{S(M)}^{D(M)}W_{\l}$ is an irreducible $A_{\alpha}(G,\S)$-module and
$$\bigoplus_{N\in \Or_M}N=\sum_{\lambda\in\Lambda_{G_M,\alpha_M}}\Ind_{S(M)}^{D(M)}W_{\l}\otimes M_{\l}$$
as an  $A_{\a}(G,{\S})\otimes V^G$-module. It is easy to see that $\dim \Ind_{S(M)}^{D(M)}W_{\l}=[G:G_M]\dim W_{\l}.$ Let $M^j$ for $j\in J$ be the orbit representatives of $\S.$ Then $\Ind_{S(M^j)}^{D(M^j)}W_{\l}$
for $j\in J$ and $\lambda\in \Lambda_{G_{M^j},\alpha_{M^j}}$ give a complete list of inequivalent irreducible $A_{\a}(G,\S)$-modules. So the associative algebra $A_{\a}(G,\S)$ and $V^G$ form a dual pair on
$$\bigoplus_{N\in \S}N=\sum_{j\in J}\sum_{\lambda\in \Lambda_{G_{M^j},\a_{M^j}}}\Ind_{S(M^j)}^{D(M^j)}W_{\l}\otimes M_{\l}^j.$$

By Theorem \ref{MTH} we have
\begin{coro}
Let $V,G,M^j$ be as before, then $\qdim_{V^G}M^j_{\lambda}=\dim({\rm Ind}_{S(M^j)}^{D(M^j)}W_{\l}) \qdim_VM^j.$
\end{coro}

From Theorem \ref{MTH1} and discussion above, we see that the $V^G$-module category is equivalent to the
$A_\a(G,\S)$-module category (here we only consider finite dimensional modules for $A_\a(G,\S)$). From
\cite{HL1}-\cite{HL3}, \cite{H}, the $V^G$-module category if a modular tensor category. So $A_{\a}(G,\S)$ should be a bialgebra and its module category should be a  modular tensor category. It is definitely interesting problem to find the coalgebra structure on $A_{\a}(G,\S)$ and define the tensor product in $A_{\a}(G,\S)$-module category. This expected tensor product will be helpful in determining the fusion rules for $V^G.$

Another problem is to how to define the intertwining operators among $g_i$-twisted $V$-modules $M^i$
where $g_i$ are any automorphism of $V$ of finite order.  If $g_1g_2=g_2g_1$ and $g_3=g_1g_2$
this was achieved in \cite{DLM1}. But it is not clear how to carry this out in general. It seems necessary to have intertwining operators among $g_i$-twisted $V$-modules $M^i$ for the purpose of studying the fusion rules for $V^G.$ As in the untwisted case, the fusion rules among twisted modules should have connection with the quantum dimensions of twisted modules.

\subsection{Global dimensions}
In this subsection we give a relation between the global dimensions of $V$ and $V^G.$

First recall from \cite{DJX} that the $\glob(V)=\sum_{M\in \mathscr{M}_V}(\qdim_VM)^2.$ It is proved in \cite{DJX} that $\glob(V)=\frac{1}{S_{V,V}^2}.$
\begin{lem}\label{lglo1} We have the following relation: $\glob(V^G)=|G|^2\glob(V).$
\end{lem}
\begin{proof}
By Theorem \ref{MTH} we know that $S_{V^G,V^G}=\frac{1}{|G|}S_{V,V}.$ Thus
$$\glob(V^G)=\frac{1}{S_{V^G,V^G}^2}=|G|^2\frac{1}{S_{V,V}^2}=|G|^2\glob(V),$$
as desired.
\end{proof}

This result has been given in \cite{ADJR} with a proof involving the category theory. The proof here uses the explicit expression of $S_{V^G,V^G}.$ Toshiyuki Abe has independently obtained this result recently.

Recall that the $M^j$ are the orbit representatives. Here is another formula about the global dimension of $V^G$ in terms of quantum dimensions of $\qdim_VM^j.$
 \begin{prop}\label{pglo2} We have
\begin{equation}\label{e1}\glob(V^G)=|G|^2\sum_{j\in J}\frac{1}{|G_{M^j}|}(\qdim_VM^j)^2,
\end{equation}
or another formula for the global dimension of $V$
\begin{equation}\label{e2}
\glob(V)=\sum_{j\in J}\frac{1}{|G_{M^j}|}(\qdim_VM^j)^2.
\end{equation}
\end{prop}
\begin{proof} It is clear from Lemma \ref{lglo1} that these two identities are equivalent. So we only need to prove the first identity. By Theorems \ref{MT} and \ref{MTH1} we know that the inequivalent irreducible $V^G$-modules are $M^j_{\l}$ for $j\in J$ and $\lambda\in \Lambda_{G_{M^j},\a_{M^j}}.$ Using Theorem \ref{MTH} we know that
$$\sum_{\lambda\in \Lambda_{G_{M^j},\a_{M^j}}}(\qdim_{V^G}M^j_{\l})^2=\frac{|G|^2}{|G_{M^j}|}(\qdim_VM^j)^2.$$
Sum over $j$ gives the desired result.
\end{proof}

The problem with (\ref{e2}) is we need to use $G_M$ which is not known in general.
We can reformulate (\ref{e2}) only using $|G|.$
Recall $\S$ is the set of all inequivalent irreducible $g$-twisted $V$-modules for $g\in G.$
\begin{coro}\label{c1} The following identity holds:
$$|G|\glob(V)=\sum_{M\in \S}(\qdim_VM)^2.$$
\end{coro}
\begin{proof} From (\ref{e2}) we see that
$$|G|\glob(V)=\sum_{j\in J}[G:G_{M^j}](\qdim_VM^j)^2.$$
Note that the $G$-orbit ${\cal O}_{M^j}$ has exactly $[G:G_{M^j}]$ irreducible twisted modules and all $N\in{\cal O}_{M^j}$
have the same character $\chi_N(\tau)=\chi_{M^j}(\tau),$ the same quantum dimension. Thus
$$[G:G_{M^j}]({\qdim_VM^j})^2=\sum_{N\in{\cal O}_{M^j}}(\qdim_VN)^2.$$
Summation over the $G$-orbits gives the result.
\end{proof}

It is worthy to mention that the formula (\ref{e2})  is true for any finite automorphism group $G$ as long as conditions (V1)-(V4) hold. It seems that this formula is new.

Corollary \ref{c1} makes us believe that  the following
 refinement of Corollary \ref{c1} is true:  Let $V$ be a rational vertex operator algebra, then  for any automorphism $g$ of finite order.
\begin{equation}\label{e3}
\glob(V)=\sum_{M\in \mathscr{M}(g)}(\qdim_VM)^2.
\end{equation}
So as far as the global dimension concerns, all the elements in the automorphism group of $V$ of finite orders have the equal weight.

\begin{prop} Let $g$ be an automorphism of $V$ of order less than or equal to $2$ and $V$ satisfy  conditions (V1)-(V4) with $G$ being the cyclic group generated by $g.$ Then ( \ref{e3}) is valid.
\end{prop}
\begin{proof}
 If $o(g)=2,$ this is clear from Corollary \ref{c1} with $G$ being the group generated by $g$. If $o(g)=3,$ there is a bijection from $\mathscr{M}(g)$ to $\mathscr{M}(g^{-1})$ by sending $M$ to its contragredient module $M'.$ It is well known that $\chi_M(\tau)=\chi_{M'}(\tau).$ So $M$ and $M'$ have the same quantum dimension. Using Corollary \ref{c1} with $G$ being the cyclic group generated by $g.$

If $o(g)=4,$ then (\ref{e3})is true for $g^2.$  Using Corollary \ref{c1} gives
 $$2\glob(V)=\sum_{M\in \mathscr{M}(g)\cup \mathscr{M}(g^3)}(\qdim_VM)^2.$$
As before we have
$$\sum_{M\in \mathscr{M}(g)}(\qdim_VM)^2=\sum_{M\in \mathscr{M}(g^3)}(\qdim_VM)^2$$
and the result follows.
\end{proof}

 We will  prove (\ref{e3}) for general $g$ in Section 5 as it involves more.

Our next example comes from the cyclic permutation orbifolds from \cite{BDM}. Let $V$ be a simple, rational, $C_2$-cofinite  vertex operator algebra of CFT type such that the weight of any irreducible module is positive except for $V$ itself. Let $k$ be a fixed positive integer. Then $V^{\otimes k}$ is a vertex operator algebra \cite{FHL} satisfying the same conditions and the symmtric group $S_k$ acts on $V^{\otimes k}$
naturally. Let  $g=(1,2,...,k).$  By \cite{ADJR},
$$\glob(V^{\otimes k})=\glob(V)^k=\frac{1}{S_{V,V}^{2k}}.$$

Let $V=M^0,M^1,...,M^p$ be the inequivalent irreducible $V$-modules. There is a
functor ${\cal T}_g^k$ from the category of $V$-modules to the category of $g$-twisted $V^{\otimes k}$-modules
such that ${\cal T}_g^k(M)=M$ as vector space. Moreover,
$$\chi_{{\cal T}_g^k(M)}(\tau)=\chi_M(\frac{\tau}{k}).$$
Note that $\chi_{V^{\otimes k}}(\tau)=\chi_V(\tau)^k.$ We now have the following standard computation
\begin{equation*}
\begin{split}
\qdim_{V^{\otimes k}}{\cal T}_g^k(M^s)=&\lim_{y\to 0}\frac{\chi_{{\cal T}_g^k(M^s)}(iy)}{\chi_{V^{\otimes k}}(\tau)}\\
=&\lim_{y\to \infty}\frac{\chi_{M^s}(-\frac{1}{iky})}{\chi_{V}(-\frac{1}{iy})^k}\\
=&\lim_{y\to \infty}\frac{\sum_{t=0}^pS_{M^s,M^t}\chi_{M^t}(iky)}{(\sum_{t=0}^pS_{V,M^t}\chi_{M^t}(iy))^k}\\
=&\frac{S_{M^s,V}}{S_{V,V}^k}
\end{split}
\end{equation*}
where we have used the fact that
$$\lim_{y\to \infty}\frac{\chi_{M^t}(iky)}{\chi_V(iy)^k}=\delta_{t,0}$$
and
$$\lim_{y\to \infty}\frac{\chi_{M^{t_1}}(iy)\cdots \chi_{M^{t_k}}(iy)}{\chi_V(iy)^k}=\delta_{(t^1,...,t^k),(0,...,0)}$$
as the weight of $M^t$ is positive except $M^0=V.$ This gives
$$\sum_{s=0}^p(\qdim_{V^{\otimes k}}{\cal T}_g^k(M^s))^2=\sum_{s=0}^p\frac{S_{M^s,V}^2}{S_{V,V}^{2k}}
=\frac{1}{S_{V,V}^{2k}}=\glob(V^{\otimes k})$$
as $\sum_{s=0}^pS_{M^s,V}^2=1$ (see \cite{DJX} and \cite{DLN}). So (\ref{e3}) is valid for $V^{\otimes k}$ and $(1,2,...,k).$
One can also verify that (\ref{e3}) is true for $V^{\otimes k}$ and any $g\in S_k$ using \cite{BDM}.

We now go back to the lattice vertex operator algebra $V_L$ with $L=\Z\alpha.$  Its irreducible modules  are $V_{L+\frac{r}{2k}\alpha}$ for $r=0,...,2k-1,$ and its irreducible $\theta$-twisted modules are $V_L^{T_i}$ for $i=0,1.$ We now compute $\qdim_{V_L}V_L^{T_i}$  directly. Notice from \cite{DJX} that $\glob(V_L)=2k$ as each irreducible
module has quantum dimension 1. Since the order of $\theta$ is two, and $\chi_{V_L^{T_0}}(\tau)=\chi_{V_L^{T_1}}(\tau)$ we see that $(\qdim_{V_L}V_L^{T_i})^2=k$ for $i=0,1$
and $\qdim_{V_L}V_L^{T_i}=\sqrt{k}.$ Of course, this is the same as before.

\section{Global dimensions and twisted modules}

Our main goal in this section is  to prove (\ref{e3}) for general $g.$ The key idea is to prove that the sum
of the squares of the quantum dimensions of irreducible $V^G$-modules appearing in irreducible $g^r$-twisted modules is $\frac{1}{T}\glob(V^G)$ for any $r=0,...,T-1$ where $G$ is the subgroup of $\Aut(V)$ generated by $g$ and $T$ is the order of $g.$

Recall that $V=\oplus_{r=0}^{T-1}V^r$ is the decomposition of $V$ into the direct sum of eigenspaces $V^r$
of $g$ with eigenvalue $e^{-2\pi ir/T}.$

\begin{lem}\label{extral} Let  $r\in\{0,...,T-1\}$ and $M$ an irreducible $g^r$-twisted $V$-module and $M_{\l}$ be an irreducible $V^G$-submodule of $M.$ Then $\frac{S_{V^1,M_{\l}}}{S_{V^G,M_{\l}}}=e^{2\pi i r/T}.$
\end{lem}
\begin{proof} The proof is similar to that of Theorem \ref{MTH}. Note that  $G_V=G.$ We have
$$Z_{V^1}(v, \tau)=\frac{1}{|G|}\sum_{r=0}^{T-1}Z_V(v, (1,g^r),\tau)e^{2\pi i r/T},$$
$$Z_{V^G}(v, \tau)=\frac{1}{|G|}\sum_{r=0}^{T-1}Z_V(v, (1,g^r),\tau).$$
Thus,
$$
Z_{V^1}(v, -1/\tau)=\frac{\tau^{\wt[v]}}{|G|}\sum_{r=0}^{T-1}\sum_{N\in \mathscr{M}(g^r)}S_{V,N}Z_N(v,\tau)e^{2\pi i r/T},$$
$$Z_{V^G}(v, -1/\tau)=\frac{\tau^{\wt[v]}}{|G|}\sum_{r=0}^{T-1}\sum_{N\in \mathscr{M}(g^r)}S_{V,N}Z_N(v,\tau).$$
Note that for any $N\in \mathscr{M}(g^r),$  $N\circ g\in \mathscr{M}(g^r).$  Using Theorems
\ref{mthm1} and  \ref{MT} immediately gives the desired result.
\end{proof}

Note from Theorem \ref{Verlinde} that $e^{2\pi i r/T}$ are the eigenvalues of $F(V^1)$ for the irreducible $V^G$-module $V^1.$

\begin{thm}\label{class}  Let $r\in\{0,...,T-1\}.$ Then
$$\sum_{X}(\qdim_{V^G}X)^2=\frac{1}{T} \glob(V^G)$$
where the sum is over the inequivalent irreducible $V^G$-modules appearing in the irreducible $g^r$-twisted $V$-modules.
\end{thm}
\begin{proof} Note from \cite{DJX}　 that for any irreducible $V^G$-module $Z$ the quantum dimension $d_Z=\qdim_{V^G}Z=\frac{S_{Z,V^G}}{S_{V^G,V^G}}.$ Also recall from \cite{DJX} that $\glob(V^G)=\frac{1}{S_{V^G,V^G}^2}.$  Thus
$$\sum_{X}d_X^2=\frac{1}{S_{V^G,V^G}^2}\sum_XS_{X,V^G}^2=\frac{1}{S_{V^G,V^G}^2}\sum_XS_{V^G,X}^2$$
where we have used the fact that $S$-matrix is symmetric (see Theorem \ref{Verlinde}). It is equivalent to show that
$$\sum_XS_{V^G,X}^2=\frac{1}{T}.$$

Set $x_r=\sum_XS_{V^G,X}^2.$ It follows from Lemmas \ref{simplefusion} and \ref{extral} we know that
$$\frac{S_{V^s,X}}{S_{V^G,X}}=(\frac{S_{V^1,X}}{S_{V^G,X}})^s=e^{2\pi irs/T}$$
for any $X$ as before and $s\in\{0,...,T-1\}.$
Using  Theorem \ref{Verlinde} and Proposition \ref{pDLN} we have orthogonal relation
$$\sum_{Z \in \mathscr{M}_{V^G}}S_{V^s,Z}S_{V^G,Z}=\delta_{s,0}$$
for any $s.$ From Theorem \ref{MTH1}, $ \mathscr{M}_{V^G}$ is a disjoint  union of $\mathscr{M}_{V^G}^r$
for $r=0,..,T-1$ where  $\mathscr{M}_{V^G}^r$ is the irreducible $V^G$-module appearing in an irreducible $g^r$-twisted $V$-module.
This gives a linear system
$$\sum_{r=0}^{T-1}x_re^{2\pi irs/T}=\sum_{r=0}^{T-1}\sum_{X\in \mathscr{M}_{V^G}^r}S_{V^G,X}^2
\frac{S_{V^s,X}}{S_{V^G,X}}=\delta_{s,0}$$
 for $s=0,...,T-1$ with non-degenerate coefficient matrix $A=(e^{2\pi irs/T})_{r,s=0}^{T-1}.$  So the linear system has a unique solution $(x_0,...,x_{T-1}).$ It is easy to see that
$x_0=x_1=\cdots =x_{T-1}=\frac{1}{T}$ is a solution. The proof is complete.
\end{proof}

Finally, we have the following result:
\begin{thm} \label{main} Let $g$ be an automorphism of $V$ of order  $n<\infty$ and $V$ satisfy  conditions (V1)-(V4) with $G$ being the cyclic group generated by $g.$ Then
 $$ \glob(V)=\sum_{M\in \mathscr{M}(g)}(\qdim_VM)^2.$$
\end{thm}
\begin{proof}  We have already discussed that for any irreducible $g$-twisted $V$-module $M,$
$G_M=G.$ That is each $M$ itself is a $G$-orbit and irreducible $V^G$-submodules appearing
in different irreducible $g$-twisted $V$-modules are inequivalent $V^G$-modules (see Theorem \ref{MT}). Note that $G$ is an abelian group and each irreducible $g$-twisted $V$-module is a direct sum of $T$ irreducible $V^G$-modules. By Theorem \ref{MTH}, for any irreducible $V^G$-submodule $M_{\l}$ of $M,$
$\qdim_{V^G}M_{\l}=\qdim_V^M.$ As a result, the sum of square of the quantum dimensions  of inequivalent
irreducible $V^G$-modules appearing in $M$ is $T(\qdim_VM)^2.$ Applying Theorem \ref{class} and Lemma \ref{lglo1}  we see that
$$\sum_{M\in  \mathscr{M}(g)}(\qdim_VM)^2=\frac{1}{T^2}\glob(V^G)=\glob(V),$$
as desired.
\end{proof}

We now give an application of Theorem \ref{main} to a holomorphic vertex operator algebra $V$ which only has one irreducible module up to isomorphism, namely $V$ itself.

\begin{prop} \label{plast}If $V$ is a holomorphic vertex operator algebra and $G$ a cyclic  group, then any irreducible $V^G$-module is a simple current.
\end{prop}
\begin{proof} From Theorem \ref{MTH}, the quantum dimension of any irreducible $g$-twisted $V$-module
is greater than or equal to one for any $g\in G.$ Using Theorem \ref{main}, for each $g$ there is a unique irreducible $g$-twisted $V$-module $V(g)$ up to isomorphism and $\qdim_VV(g)=1.$  This implies that
$G_{V(g)}=G.$ According to a well known fact that any projective representation of a cyclic group is ordinary, we see that $V(g)$ is a $G$-module.  By decomposition (\ref{decom}) we have
$$V(g)=\oplus_{\lambda\in \irr(G)}W_{\l}\otimes V(g)_\l$$
where $\irr(G)$ is the set of irreducible characters if $G.$ Moreover, each $V(g)_{\l}\ne 0.$
Using Theorem \ref{MTH} concludes
that $\qdim_{V^G}V(g)_{\l}=\qdim_VV(g)=1$ for any irreducible $V^G$-module $V(g)_{\l}$ appearing in $V(g).$ It follows from \cite{DJX} that $V(g)_{\l}$ is a simple current.
\end{proof}

\begin{rem} Proposition \ref{plast} is false if $V$ is not holomorphic. Here is a counter example. Consider the lattice vertex operator algebra $V_L$ with $L=\Z\alpha$ and $(\alpha,\alpha)=2k.$ Let $G$ be the cyclic group generated by $theta$ which is the automorphism of $V_L$ induced from the $-1$-isometry of $L.$
We have already known that $V_L$ has irreducible $\theta$-twisted modules $V_L^{T_i}$ with $i=0,1$ with quantum dimension $\sqrt{k}.$ Each $V_L^{T_i}$ is a direct sum of two irreducible $V_L^G$-modules with quantum dimension  $\sqrt{k}$ by Theorem \ref{MTH}. If $k>1$ these irreducible $V_L^G$-modules appearing in irreducible $\theta$-twisted $V_L$-modules are not simple current although $G$ is abelian.
\end{rem}

We next compute the fusion rules for the vertex operator algebra $V^G$ in the setting of Proposition \ref{plast}. For this purpose we need to recall the quantum double $D(G)$ associated to the group $G$ from
\cite{D}, \cite{DPR}.  The quantum double $D(G)$  is the Hopf algebra with underlying vector space
$\C[G]\otimes \C[G]^*,$ multiplication,
$$(x\otimes e(g))(y\otimes e(h)) = \delta_{g,h}xy\otimes e(g)$$
and coproduct
$$\Delta(x\otimes e(g)) = \sum_{hk=g}(x\otimes e(h))(x\otimes e(k)) .$$
Clearly, $D(G)$ is exactly the $A_{\alpha}(G,{\cal S})$ as algebras defined  before with trivial $\alpha.$
We can index the irreducible representations by $( \lambda, e(g))$ for $g\in G$ and $\l\in \irr(G)$ such that
$x\otimes e(h)$ acts as  $\delta_{g,h}\lambda(x).$ Observe that the tensor product of irreducible
$D(G)$-modules obeys the following rule:
$$( \lambda, e(g))\otimes ( \mu, e(h))=(\lambda\mu, e(gh)).$$

\begin{prop} Let  $V$  be a holomorphic vertex operator algebra and $G$ a cyclic group. For $g,h\in G$ and $\lambda,\mu\in\irr(G)$ we have
$$V(g)_{\l}\boxtimes V(h)_{\mu}=V(gh)_{\l\mu}.$$
That is , the $V^G$-module tensor category and $D(G)$-module tensor category are equivalent.
\end{prop}\begin{proof} Note that each irreducible $V^G$-module  $V(g)_\l$ is  a $D(G)$-module
which is a direct sum of $D(G)$-module $(\lambda,e(g)).$ So $V(g)_{\l}\boxtimes V(h)_{\mu}$ is also a  $D(G)$-module. By Proposition \ref{plast}, $V(g)_{\l}\boxtimes V(h)_{\mu}$ is an irreducible $V^G$-module
as $V(g)_{\l}$ is a simple current.  It is good enough to show that $V(g)_{\l}\boxtimes V(h)_{\mu}$ is a $D(G)$--module with character $(\l\mu,e(gh)).$

Recall that the tensor product $V(g)_{\l}\boxtimes V(h)_{\mu}$ is a $V^G$-module
together with an intertwining operator ${\cal Y}$ of type $\left(
\begin{array}{c}
\ \ V(g)_{\l}\boxtimes V(h)_{\mu}  \\
V(g)_{\l} \ \  V(h)_{\mu}
\end{array}
 \right)$
such that for any $V^G$-module $W$ and any intertwining operator $I$ of type $\left(
\begin{array}{c}
\ \  W  \\
V(g)_{\l} \ \  V(h)_{\mu}
\end{array}
 \right)$
there is a unique $V^G$-module homomorphism $\phi: V(g)_{\l}\boxtimes V(h)_{\mu}\to W$ so that
$I=\phi {\cal Y}.$  We have already mentioned that $D(G)= A_{\alpha}(G,{\cal S})$ and the actions of $V^G$ and $D(G)$ commute on $V(g)$ for all $g\in G.$ Note  that  $I\left(
\begin{array}{c}
\ \ V(g)_{\l}\boxtimes V(h)_{\mu}  \\
V(g)_{\l} \ \  V(h)_{\mu}
\end{array}
 \right)$
is  one dimensional as $V(g)_{\l}\boxtimes V(h)_{\mu}$ is irreducible.

We now define an action of $D(G)$ on the space $I\left(
\begin{array}{c}
\ \ V(g)_{\l}\boxtimes V(h)_{\mu}  \\
V(g)_{\l} \ \  V(h)_{\mu}
\end{array}
 \right)$ such that
$${\cal Y}_{x\otimes e(k)}(u,z)w=\sum_{ab=k}{\cal Y}((x\otimes e(a))u,z)(x\otimes e(b))w$$
for $x,k\in G,$  $u\in V(g)_{\l}$ and $w\in V(h)_{\mu}.$
It is easy to verify that $${\cal Y}_{x\otimes e(k)}\in I\left(
\begin{array}{c}
\ \ V(g)_{\l}\boxtimes V(h)_{\mu}  \\
V(g)_{\l} \ \  V(h)_{\mu}
\end{array}
 \right)$$ and
$I\left(
\begin{array}{c}
\ \ V(g)_{\l}\boxtimes V(h)_{\mu}  \\
V(g)_{\l} \ \  V(h)_{\mu}
\end{array}
 \right)$ is one dimensional irreducible $D(G)$-module.
It is also clear that
$${\cal Y}_{x\otimes e(k)}(u,z)w=\delta_{k,gh}{\cal Y}((x\otimes e(g))u,z)(x\otimes e(h))w=\lambda(x)\mu(x){\cal Y}(u,z)w.$$
Thus, $I\left(
\begin{array}{c}
\ \ V(g)_{\l}\boxtimes V(h)_{\mu}  \\
V(g)_{\l} \ \  V(h)_{\mu} \end{array}
 \right)$
is isomorphic to $(\lambda\mu,e(gh)).$
Now consider the linear map from  $I\left(
\begin{array}{c}
\ \ V(g)_{\l}\boxtimes V(h)_{\mu}  \\
V(g)_{\l} \ \  V(h)_{\mu} \end{array}
 \right)$ to  $V(g)_{\l}\boxtimes V(h)_{\mu} $ such that ${\cal Y}$ is mapped to $u_nw$
for any nonzero $u\in V(g)_{\l}$ and $w\in V(h)_{\mu}$ and $n\in \Q.$ This map is a $D(G)$-module
homomorphism. By \cite{DL1}, there exist $u, w$ and $n$ such that $u_nw\ne 0.$ So $V(g)_{\l}\boxtimes V(h)_{\mu} $ has a $D(G)$-submodule isomorphic to $(\lambda\mu,e(gh))$ and contains $V(gh)_{\l\mu}.$
As a result,  $V(g)_{\l}\boxtimes V(h)_{\mu} = V(gh)_{\l\mu}.$ \end{proof}

For general $V$ the fusion rules are more complicated.  However, we can determine the fusion rules among irreducible $V^G$-modules appearing in $V$ using \cite{T}.
\begin{coro} Let $V$ be a vertex operator algebra and $G$ a finite automorphism group of $V$
satisfying conditions (V1)-(V4).  Then for any $\lambda,\mu\in\irr(G)$
$$V_{\l}\boxtimes V_{\mu}=\sum_{\gamma\in \irr(G)}\dim {\rm Hom}_{G}(W_{\l}\otimes W_{\mu},W_{\gamma})
V_{\gamma}.$$
\end{coro}
\begin{proof} From \cite{DJX} we know that
$$\qdim_{V^G}V_{\l}\boxtimes V_{\mu}=\qdim_{V^G}V_\l\qdim_{V^G}V_{\mu}=\dim W_{\l}\dim W_{\mu}.$$
From \cite{T},  $N_{V_{\l}, V_{\mu}}^{V_{\gamma}}\geq \dim \Hom_{G}(W_{\l}\otimes W_{\mu},W_{\gamma}).$ The proof is complete.
\end{proof}

\end{document}